\documentclass[10pt,twoside]{article}

\usepackage{amssymb}
\usepackage{amsbsy} 
\usepackage{amstext}
\usepackage{enumerate}
\usepackage{verbatim} 
\usepackage{graphicx}

\usepackage{theorem}
\newtheorem{theor}{Theorem}[section]
\newtheorem{prop}[theor]{Proposition}
\newtheorem{coro}[theor]{Corollary}
\newtheorem{lemma}[theor]{Lemma}

\newtheorem{remark}[theor]{Remark}

\newenvironment{proof}
	{\par {\bf Proof:}}
 	{\hfill $\square$ \medskip}

\newcommand{\esup}{\mathop{\mbox{\rm ess\,sup}}}
\newcommand{\einf}{\mathop{\mbox{\rm ess\,inf}}}

\def\mc{\mathcal}
\def\ds{\displaystyle}
\def\be{\begin{equation}}
\def\ee{\end{equation}}
\def\F{{\mathcal F}}

\def\G{{\mathcal G}}

\def\H{{\mathcal H}}

\def\NN{{\mathfrak N}}
\def\RR{{\mathcal R}}

\def\N{\mathbb N}
\def\Z{\mathbb Z}
\def\R{\mathbb R}
\def\C{\mathbb C}
\def\D{\mathbb D}
\def\I{\mathbb I}
\def\J{\mathbb J}

\def\om{\omega}
\def\<{\langle}
\def\>{\rangle}

\def\bu{\mathbf u}
\def\bv{\mathbf v}

\def\bof{\mathbf f}

\title{Wavelet frames: Spectral techniques and extension principles}
\author{F. G\'omez-Cubillo$^1$, S. Villullas$^2$}

\begin{document}

\maketitle
\begin{center}
{\small $^1$Dpto de An\'alisis Matem\'atico, IMUVa, Universidad de Valladolid, Facultad de Ciencias, 47011 Valladolid, Spain. fgcubill@am.uva.es.} 

{\small $^2$Dpto de Econom\'{\i}a, Universidad Carlos III de Madrid,
C/Madrid 126, 28903 Getafe (Madrid), Spain. svillull@uc3m.es.}
\end{center}

\begin{abstract}
This work characterizes (dyadic) wavelet frames for $L^2({\mathbb R})$ by means of spectral techniques. These techniques use decomposability properties of the frame operator in spectral representations associated to the dilation operator. 
The approach is closely related to usual Fourier domain fiberization techniques, dual Gramian analysis and extension principles, which are described here on the basis of the periodized Fourier transform.
In a second paper of this series, we shall show how the spectral formulas obtained here permit us to calculate all the tight wavelet frames for $L^2(\R)$ with a fixed number of generators of minimal support. 


\end{abstract}

\section{Introduction}\label{s1}

Let $\H$ be a separable Hilbert space with norm $||\cdot||_\H$ and scalar product $\<\cdot,\cdot\>_\H$ (linear in the first component and conjugate-linear in the second).  A countable subset $X$ of $\H$ is called a {\bf frame} for $\H$ if there exist constants $A,B>0$ such that the following inequalities hold:
\be\label{c5.2}
A\,||f||^2\leq \sum_{x\in X} |\<f,x\>_\H|^2\leq B\,||f||^2,\quad (f\in\H).
\ee
It is well known (see, e.g., \cite[Chapter 5]{Chr}) that $X$ is a frame for $\H$ if and only if the corresponding {\it synthesis operator}
\be\label{c5.3}
T_X:l^2(X)\to \H,\quad T_X\{c_x\}_{x\in X}:=\sum_{x\in X} c_x\,x
\ee
is well-defined and bounded from $l^2(X)$ onto $\H$. In such case, the {\it analysis operator} is the adjoint operator of $T_X$, given by
$T_X^*:\H\to l^2(X)$, $T_X^*f=\{\<f,x\>_\H\}_{x\in X}$,
and the {\bf frame operator} $S=T_XT_X^*$,
\be\label{c5.5}
S:\H\to \H,\quad Sf=T_XT_X^*f=\sum_{x\in X}\<f,x\>_\H\,x\,,
\ee
is bounded, positive and invertible. Furthermore, $S^{-1}X=\{S^{-1}x:x\in X\}$ is also a frame, called the {\it canonical dual frame} of $X$, and the ``perfect reconstruction formula'' 
\be\label{c5.7}
f=\sum_{x\in X} \<f,S^{-1}x\>_\H\,x\,,\quad (f\in\H)\,,
\ee
is satisfied. In (\ref{c5.5}) and (\ref{c5.7}) the series converge unconditionally for all $f\in\H$, i.e., for every permutation of the summands the resulting series is convergent.
For a frame $X$, the sharpest possible constants $A,B$ in (\ref{c5.2}) are $A=||S^{-1}||^{-1}$ and $B=||S||=||T_X||^2=||T_X^*||^2$ and are usually referred to as the {\it frame bounds} . 
A frame $X$ is called a {\bf tight frame} if its frame bounds coincide. 

In this work we focus attention on $\H=L^2(\R)$ and a special type of subsets, the {\bf (dyadic) wavelet systems} of $L^2(\R)$ of the form
\be\label{ws}
X=X_\Psi:=\big\{\psi_{k,j}:=D^kT^j\psi: \psi\in\Psi,\,k,j\in\Z\big\},
\ee
where $\Psi$ is a finite or countable family of $L^2(\R)$, and $T$ and $D$ are the translation and (dyadic) dilation operators on $L^2(\R)$ defined by 
\be\label{tdo}
[Tf](x):=f(x-1)\,,\quad [ D f](x):=2^{1/2}\,f(2x)\,,\quad  (f\in L^2(\R),\,x\in\R)\,.
\ee 
Wavelet systems of the form (\ref{ws}) which are frames for $L^2(\R)$ are called {\bf wavelet frames} for $L^2(\R)$ with {\it generator set} $\Psi$.

Wavelet frames for $L^2(\R^d)$ have been extensively studied in the last three decades. One of the main lines of study involves fiberization techniques \cite{RS95,RS97b,RS97a,RS98,RS99,DRS03,RS05}, translation (shift) invariant subspaces \cite{BVR94b,RS95} and refinable functions \cite{BVR93,BeLi98}, leading to the {\it unitary extension principle} (UEP) \cite{RS97b,RS97a,CH00,BT01}, subsequently extended in the form of {\it oblique extension principle} (OEP) \cite{CHS02,DRS03,H10,AMS14,APS16} and {\it duality principle} \cite{FHS16,FJS16}. See the cited references for an exhaustive overview.

Here we present an alternative approach to study wavelet frames for $L^2(\R)$, an approach already introduced in \cite{GS11b,GS11a,GSV12,GV18} for the analysis of orthonormal wavelets.
The adjective ``spectral'' for the techniques we develop comes from the use of suitable spectral representations for the translation and dilation operators $T$ and $D$. Such spectral representations are given in section \ref{sect2}. 

Section \ref{sect3} contains the main results of the work.
Theorem \ref{th1} characterizes the wavelet systems $X$ of the form (\ref{ws}) that are frames and tight frames for $L^2(\R)$. The result is based on the fact that the frame operator $S=T_XT_X^*$ is decomposable (diagonal for tight frames) in the direct integral associated to the spectral representation of the dilation operator $D$.
The matrix elements and fibers of the decomposable expression of $S$ are given in theorem \ref{th2}. Theorems \ref{th1} and \ref{th2} lead to an affordable description of tight wavelet frames in corollary \ref{coro10}.
Let us note that the usual fiberization techniques work on the Fourier domain, i.e., on a spectral representation of the translation operator $T$, in spite of wavelet systems of the form (\ref{ws}) are dilation invariant but not translation (shift) invariant. In order to have a decomposable expression of the frame operator in the Fourier domain, it is necessary to extend the wavelet system to a shift invariant one, the so-called {\it quasi-affine system}; see remark \ref{rm7} for details. 

In section \ref{s5} we translate some usual concepts and results in the theory of wavelet frames to the spectral framework. This is done by means of the {\it periodized Fourier transform} given in proposition \ref{ptft1}. So, for example, the {\it bracket product} used dealing with shift-invariant systems is actually a scalar product (see remark \ref{rm39}) and every element of the principal shift-invariant subspace generated by a function $\phi\in L^2(\R)$ is in fact ``colineal to $\phi$" (proposition \ref{psis}). These results are of interest dealing with refinable functions $\phi$ and associated multiresolution analyses to derive {\it extension principles}. A version of the UEP is given in theorem \ref{tuep} and versions of the OEP can be found in corollary \ref{ll45cx} and theorem \ref{oepams}.

A second part of this work \cite{GV19-2} shall show how corollary \ref{coro10} can be used  to obtain all the tight wavelet frames for $L^2(\R)$ with a fixed number of generators $\psi\in\Psi$ of minimal support. 
Like in \cite{GS11b}, Hardy classes of vector-valued functions and operator-valued inner functions will play a central role.

\section{Spectral representations for $T$ and $D$}\label{sect2}

We begin by introducing spectral representations for the translation and dilation operators, $T$ and $D$,  defined on $L^2(\R)$ by (\ref{tdo}). These representations have already been considered in \cite{GS11b,GS11a,GSV12,GV18}. They live in spaces of the form $L^2(\partial {\D};\H)$ we next define. 

Let ${\D}$ denote the open unit disc of the complex plane $\C$ and $\partial {\D}$ its boundary:
$$
{\D}:=\{\lambda\in\C:|\lambda|<1\}\,,\qquad \partial {\D}:=\{\omega\in\C:|\omega|=1\}\,.
$$
In $\partial {\D}$ interpret measurability in the sense of Borel and consider the normalized Lebesgue measure $d\omega/(2\pi)$.
Given a separable Hilbert space $\H$, let $L^2(\partial {\D};\H)$ denote the set of all measurable functions $\bv:\partial {\D}\to \H$ such that 
$$\int_{\partial {\D}} ||\bv(\omega)||^2_\H\,\frac{d\omega}{2\pi}<\infty$$ (modulo sets of measure zero); measurability here can be interpreted either strongly or weakly, which amounts to the same due to the separability of $\H$. The functions in $L^2(\partial {\D};\H)$ constitute a Hilbert space with pointwise definition of linear operations and inner product given by 
$$
\<\bu,\bv\>_{L^2(\partial {\D};\H)}:=\int_{\partial {\D}} \<\bu(\omega),\bv(\omega)\>_\H\,\frac{d\omega}{2\pi}\,,\qquad \big(\bu,\bv\in L^2(\partial {\D};\H)\big)\,.
$$
The space $L^2(\partial {\D};\H)$ is a particular case of direct integral of Hilbert spaces; see \cite[Chapter 14]{KRII97} for details.

A bounded operator $S:L^2(\partial {\D};\H)\to L^2(\partial {\D};\H)$ is said to be {\it decomposable} when there is a function $\om\mapsto S(\om)$ on $\partial {\D}$ such that $S(\om):\H\to\H$ is a bounded operator and for each $\bu\in L^2(\partial {\D};\H)$, $S(\om)\bu(\om)=[S\bu](\om)$ for almost every (shortly, a.e.) $\om \in \partial {\D}$. For a decomposable operator $S$ we shall write
$$
S=S(\om)\,.
$$
If, in addition, $S(\om)=s(\om)I_\H$, where $I_\H$ is the identity operator on $\H$ and $s:\partial {\D}\to\C$ is a measurable function, we say that $S$ is {\it diagonalizable} and write $S=s(\om)I_\H$. 

For the sake of completeness, the following proposition includes two well-known results on operator theory. Definitions and terminology can be found in the references cited in the proof.

\begin{prop}\label{prop1vn}
Let $S:L^2(\partial {\D};\H)\to L^2(\partial {\D};\H)$ be a bounded operator.
\begin{itemize}
\item[(i)]
A bounded operator $S:L^2(\partial {\D};\H)\to L^2(\partial {\D};\H)$ commutes with every diagonalizable operator if and only if $S$ commutes with the diagonalizable operator $\om I_\H$.
\item[(ii)]
The set of decomposable operators in $L^2(\partial {\D};\H)$ is a von Neumann algebra with abelian commutant coinciding with the family of diagonalizable operators.
\end{itemize}
\end{prop}

\begin{proof}
(i) is a consequence of the spectral theory for unitary operators (of constant multiplicity): see, for example, theorems 5.4.8, 6.2.4 and 7.2.1 in \cite{BS}.
(ii) is a particular case of \cite[th.14.1.10]{KRII97}.  
\end{proof}
 
The spectral representations of $T$ and $D$ we consider are given in propositions \ref{ptft} and \ref{pdft}, respectively, where $T$ and $D$ are transformed into diagonalizable operators of the form $\om I_\H$ on suitable spaces $L^2(\partial {\D};\H)$. Proofs and more details can be found in \cite{GS11a}. 
We begin by considering an orthonormal basis (shortly, ONB) $\{L_{i}^{(0)}(x)\}_{i\in\I}$  of $L^2[0,1)$ and ONBs $\{K_{\pm,j}^{(0)}(x)\}_{j\in\J}$ of $L^2[\pm 1,\pm 2)$, where $\I$, $\J$ are denumerable sets of indices (usually, $\N$, $\N\cup\{0\}$ or $\Z$). Obviously, the families
\be\label{lt}
\big\{L_{i}^{(n)}(x):=[T^n L_{i}^{(0)}](x)=L_{i}^{(0)}(x-n)\big\}_{i\in\I,n\in\Z},
\ee
\be\label{kd}
\big\{K_{s,j}^{(m)}(x):=[D^m K_{s,j}^{(0)}](x)=2^{m/2}K_{s,j}^{(0)}(2^mx)\big\}_{j\in\J,m\in\Z,s=\pm}
\ee
are ONBs of $L^2(\R)$ and, for each $f\in L^2(\R)$, one has (in $L^2$-sense)
\be\label{fdobl2}
f=\sum_{i,n} \hat f^{(n)}_{i} L_i^{(n)},\text{ with }\hat f^{(n)}_{i}:=\<f,L_i^{(n)}\>_{L^2(\R)}\,,
\ee
\be\label{fdobl1}
f=\sum_{s,j,m} \tilde f^{(m)}_{s,j} K_{s,j}^{(m)},\text{ with }\tilde f^{(m)}_{s,j}:=\<f,K_{s,j}^{(m)}\>_{L^2(\R)}\,.
\ee
In what follows, fixed ONBs $\{L_{i}^{(n)}(x)\}_{i\in\I,n\in\Z}$ and  $\{K_{s,j}^{(m)}(x)\}_{j\in\J,m\in\Z,s=\pm}$ of $L^2(\R)$ as above, for each $f\in L^2(\R)$ we shall write
$$
f=\big\{\hat f_{i}^{(n)}\big\}=\big\{\tilde f_{s,j}^{(m)}\big\}\,.
$$

A spectral representation for the dilation operator $D$ on $L^2(\R)$ is given in the next result. Here, $l^2(\J)$ denotes the Hilbert space of sequences of complex numbers $(c_j)_{j\in\J}$ such that $\sum_{j\in\J} |c_j|^2<\infty$, $\big\{u_{s,j}\big\}_{j\in\J,s=\pm}$ is a fixed ONB of $l^2(\J)\oplus l^2(\J)$ and $\oplus$ denotes orthogonal sum.

\begin{prop}{\rm\cite[Proposition 1]{GS11a}}\label{pdft}
The operator $\G$ defined by
$$
\begin{array}{rccl}
\G: & L^2(\R) & \longrightarrow & \ds L^2\big(\partial {\D};l^2(\J)\oplus l^2(\J)\big) 
\\[3ex]
& f & \mapsto & \ds \tilde \bof:=\bigoplus_{s=\pm}\bigoplus_{j\in\J} \left[ \sum_{m\in\Z} \om^{m}\,\tilde f_{s,j}^{(m)} \right]\,u_{s,j}\,.  
\end{array}
$$
determines a spectral model for the dilation operator $D$, i.e., $\G$ is unitary and 
$$
\G  D\G^{-1}=\om\,I_{l^2(\J)\oplus l^2(\J)}\,.
$$
\end{prop}

Now, for the translation operator $T$ on $L^2(\R)$, if $\big\{u_{i}\big\}_{i\in\I}$ is a fixed ONB of $l^2(\I)$, one has:

\begin{prop}{\rm \cite[Proposition 3]{GS11a}}\label{ptft}
The operator $\F$ given by
$$
\begin{array}{rccl}
\F: & L^2(\R) & \longrightarrow & \ds L^2\big(\partial {\D};l^2(\I)\big)
\\[3ex]
& f & \mapsto & \hat \bof:=\ds \bigoplus_{i\in\I} \left[ \sum_{n\in\Z} \om^{n}\,\hat f_{i}^{(n)} \right]\,u_i\,,
\end{array}
$$
determines a spectral model for the translation operator $T$, i.e., $\F$ is unitary and 
$$
\F T \F^{-1}=\om\,I_{l^2(\I)}\,.
$$
\end{prop}

In the sequel we shall write
$$
\hat{f}_{i}(\om):=\sum_{n\in\Z} \om^{n}\,\hat f_{i}^{(n)}\quad \text{and}\quad
\tilde{f}_{s,j}(\om):=\sum_{m\in\Z} \om^{m}\,\tilde f_{s,j}^{(m)}\,,\quad (f\in L^2(\R);i\in\I;s=\pm,j\in\J)\,.
$$

The change of representation between both expansions (\ref{fdobl2}) and (\ref{fdobl1}) is governed by a matrix $\big(\alpha_{i,n}^{s,j,m}\big)$, where
\be\label{chm}
\alpha_{i,n}^{s,j,m}:=\< L_i^{(n)},K_{s,j}^{(m)}\>_{L^2(\R)}\,.
\ee
A useful identity shall be
\be\label{ee2}
[\widetilde{L_i^{(n)}}]_{s,j}(\om)=\sum_{m} \om^m\,[\widetilde{L_i^{(n)}}]_{s,j}^{(m)}=\sum_{m} \om^m\,\< L_i^{(n)},K_{s,j}^{(m)}\>_{L^2(\R)}=\sum_{m} \om^m\,\alpha_{i,n}^{s,j,m}\,.
\ee

Finally, let $\hat f$ denote the usual {\bf Fourier transform} of a function $f\in L^2(\R)$:
\be\label{FT}
\hat f(y):=\int_\R f(x)\,e^{-2\pi ixy}\,dx\,,\quad (y\in\R)\,.
\ee

\begin{prop}[Periodized Fourier transform]{\rm \cite[Proposition 6]{GS11a}}\label{ptft1}
Let $\big\{u_{k}\big\}_{k\in\Z}$ be a fixed ONB of $l^2(\Z)$ and $\F_*$ the operator defined by  
\be\label{tft1}
\begin{array}{rccl}
\F_*: & L^2(\R) & \longrightarrow & \ds L^2\big(\partial \D;l^2(\Z)\big)
\\[2ex]
& f & \mapsto & \ds \hat \bof_*:=\bigoplus_{k\in\Z} \,\hat f_k(\om)\,u_k\,,
\end{array}
\ee
where, if $\om=e^{2\pi i\theta}$, 
\be\label{tft1a}
\hat f_k(\om)=\hat f_k(e^{2\pi i\theta}):=\overline{\hat f(\theta+k)}\,,\quad \text{ for a.e. }\theta\in[0,1)\text{ and }k\in\Z\,.
\ee
Then $\F_*$ determines a functional spectral model for the translation operator $T$, that is, $\F_*$ is a unitary operator such that
\be\label{ditf}
[\F_* Tf](\om)=\om\cdot[\F_*f](\om)\,,\quad \text{ for a.e }\om\in\partial {\D}\text{ and }f\in L^2(\R)\,.
\ee 
\end{prop}

\section{Spectral techniques for wavelet frames}\label{sect3}

Although the results of this section can be given in the more general setting of affine systems in $L^2(\R^d)$, we restrict attention on the Hilbert space $L^2(\R)$ and {wavelet systems} $X$ of the form (\ref{ws}).
Theorems \ref{th1} and \ref{th2} and corollary \ref{coro10} below work on the dilation representation of proposition \ref{pdft} and not on the usual Fourier domain of fiberization techniques (see, e.g., \cite[th.3.3.5]{RS95} and \cite[th.3.1]{RS97a}). 
Some comments comparing spectral with fiberization techniques are included in remark \ref{rm7}.
First, we include a technical result necessary to prove theorem \ref{th1}:

\begin{lemma}\label{lth1}
Let $X$ be a wavelet system in $L^2(\R)$ of the form (\ref{ws}) and such that 
\be\label{aePsi}
\sup_{\psi\in\Psi} ||\psi||_{L^2(\R)}=M<\infty\,.
\ee 
Then the corresponding operator $T_X$, given by (\ref{c5.3}), is a well-defined bounded operator from $l^2(X)$ into $L^2(\R)$ if and only if the corresponding operator $S$, given by (\ref{c5.5}), is a well-defined bounded operator on $L^2(\R)$.
\end{lemma}

\begin{proof}
If $T_X$ is a well-defined bounded operator from $l^2(X)$ into $L^2(\R)$, then  the adjoint $T_X^*$ and $S=T_XT_X^*$ are also well-defined bounded operators. 
For the opposite implication, it is obvious that the synthesis operator $T_X$, given by (\ref{c5.3}), is well-defined at least on the dense subspace $l^{00}(X)$ of $l^2(X)$ formed by the sequences $\{c_x\}\in l^2(X)$ with a finite number of non-zero components. In principle, consider  $T_X$ defined on $l^{00}(X)$. Such operator $T_X$ is {\it preclosed} if and only if (\ref{aePsi}) is satisfied.
To see this, recall that $T_X$ is preclosed if and only if for every sequence $\{c_n\}\subset l^{00}(X)$ such that $\lim_{n\to\infty} c_n=0$ one has $\lim_{n\to\infty} T_X c_n=0$ (see, for example, \cite[page 155]{KRI97}).
If (\ref{aePsi}) is satisfied and $\lim_{n\to\infty} c_n=0$, then 
$\lim_{n\to\infty} ||T_X c_n||_{L^2(X)}\leq M\,\lim_{n\to\infty} ||c_n||_{l^2(X)}=0$. Conversely, if (\ref{aePsi}) is not satisfied, consider a sequence $\{\psi_n\}\subset\Psi$ such that $\lim_{n\to\infty} ||\psi_n||_{L^2(\R)}=\infty$ and a sequence $\{c_n\}\subset l^{00}(X)$ such that the only non-zero element of $c_n$ is the $n$-component with modulus equal to $||\psi_n||^{-1}_{L^2(\R)}$ for $n\geq n_0$.
Moreover, if $T_X$ is preclosed with closure $\overline{T_X}$, the adjoint $T_X^*$ is defined in a dense domain ${\mc D}(T_X^*)$ of $L^2(\R)$ and $T_X^*$ is a closed operator, $T_X^*=[\overline{T_X}]^*$, $T_X^{**}=\overline{T_X}$ and the domain of $\overline{T_X}T_X^*$, ${\mc D}(\overline{T_X}T_X^*)$, is a {\it core} for $T_X^*$; see remark 2.7.7 and theorem 2.7.8 in \cite{KRI97}. 
Assume that $\overline{T_X}T_X^*$ is bounded on ${\mc D}(\overline{T_X}T_X^*)$ and, then, $\overline{T_X}T_X^*$ admits a well-defined bounded extension $S$ on $L^2(\R)$. Since, for $h\in {\mc D}(\overline{T_X}T_X^*)$,
$$
|\<Sh,h\>_{L^2(\R)}|=|\<\overline{T_X}T_X^*h,h\>_{L^2(\R)}|=\<T_X^*h,T_X^*h\>_{L^2(\R)}=||T_X^*h||_{l^2(X)}^2\leq ||S||\,||h||_{L^2(\R)}^2\,,
$$
$T_X^*$ is also bounded on ${\mc D}(\overline{T_X}T_X^*)$, so that $T_X^*$ admits a well-defined bounded extension on $L^2(\R)$ whose (well-defined bounded) adjoint extends $T_X$ and $\overline{T_X}$. Thus, $T_X$, defined in principle on $l^{00}(X)$, admits a well-defined bounded extension on $l^2(X)$.
\end{proof}

Recall that a countable subset $X$ of a Hilbert space $\H$ is called a {\bf Bessel system} if the second inequality in (\ref{c5.2}) is satisfied for some constant $B>0$. In such case, every number $B$ satisfying (\ref{c5.2}) is called a {\it Bessel bound} for $X$.

\begin{theor}\label{th1}
Let $X$ be a wavelet system in $L^2(\R)$ of the form (\ref{ws}) and such that (\ref{aePsi}) is satisfied. Then:
\begin{enumerate}
\item
$X$ is a Bessel system if and only the corresponding operator $S$, given by (\ref{c5.5}), is a well-defined bounded operator on $L^2(\R)$. In such case, $S$  commutes with $D$ and, then, going to the dilation representation given in proposition \ref{pdft}, $\G S\G^{-1}$ is a decomposable operator on $L^2(\partial {\D}, l^2(\J)\oplus l^2(\J))$:
$$
\G S\G^{-1}=S(\om)\,.
$$
Moreover, $S$ is positive, $S(\om)$ is positive for a.e. $\om\in \partial {\D}$ and 
$$
||S||=\esup_{\om\in\partial\D} ||S(\om)||=\esup_{\om\in\partial\D} \sup_{||u||_{l^2(\J)\oplus l^2(\J)}=1}||S(\om)u||_{l^2(\J)\oplus l^2(\J)}<\infty\,.
$$

\item
$X$ is a {frame} for $L^2(\R)$ if and only if $S$ is a well-defined bounded operator on $L^2(\R)$ with bounded two-sided inverse $S^{-1}$. In such case, $S^{-1}$ also commutes with $D$ and
$$
\G S^{-1}\G^{-1}=S(\om)^{-1}\,.
$$
Equivalently, $X$ is a {frame} for $L^2(\R)$ if and only if
$$
\alpha:=\esup_{\om\in\partial\D} \sup_{||u||_{l^2(\J)\oplus l^2(\J)}=1}||S(\om)u||_{l^2(\J)\oplus l^2(\J)}<\infty
$$
and
$$
\beta:=\einf_{\om\in\partial\D}\inf_{||u||_{l^2(\J)\oplus l^2(\J)}=1} ||S(\om)u||_{l^2(\J)\oplus l^2(\J)}>0\,;
$$
in such case, $||S||=\alpha$ and $||S^{-1}||=\beta^{-1}$.

\item
$X$ is a {tight frame} for $L^2(\R)$, with frame bound $B$, if and only if $S=B\,I_{L^2(\R)}$ or, equivalently,
$$
\G S\G^{-1}=B\,I_{l^2(\J)\oplus l^2(\J)}\,.
$$
\end{enumerate}
\end{theor}

\begin{proof}
1. It is well known \cite[th.3.2.3]{Chr} that a countable subset $X$ of a Hilbert space $\H$ is a Bessel system with Bessel bound $B$ if and only if the corresponding synthesis operator $T_X$, given by (\ref{c5.3}), is a well-defined bounded operator and $||T_X||\leq B^{1/2}$. 
By lemma \ref{lth1}, $T_X$ is a well-defined bounded operator if and only if the corresponding operator $S$, given by (\ref{c5.5}), is a well-defined bounded operator on $L^2(\R)$.
In such case, $||T_X^*||=||T_X||$, and $S$ is positive since $\<Sf,f\>_{L^2(\R)}=\<T_X^*f,T_X^*f\>_{l^2(X)}\geq 0$.
Moreover, $S$ commutes with the dilation operator $D$:
$$
SD f=\sum_{\stackrel{k,j\in\Z}{\psi\in\Psi}}\<Df,D^{k}T^j\psi\>D^kT^j\psi=
\sum_{\stackrel{k,j\in\Z}{\psi\in\Psi}}\< f,D^{k-1}T^j\psi\>D^kT^j\psi=DSf\,,\quad (f\in L^2(\R))\,.
$$
(We use that $D$ is unitary, $D^*=D^{-1}$, and that the series defining $S$ converges unconditionally for all $f\in L^2(\R)$; see \cite[Corollary 3.2.5]{Chr} and lemma \ref{lcu} below.)
By proposition \ref{prop1vn}, this implies that $\G S\G^{-1}$ is a decomposable operator on $L^2(\partial {\D}, l^2(\J)\oplus l^2(\J))$.
Since $S$ is positive, $S(\om)$ is positive a.e. (see \cite[prop.14.1.8-9]{KRII97}).
Being $\G$ unitary, $||S||=||\G S\G^{-1}||$ and, by \cite[Prop.14.1.9]{KRII97}, 
$$
||S||=||\G S\G^{-1}||=\esup_{\om\in\partial\D} ||S(\om)||\,.
$$

2. In terms of $S$, the inequalities in (\ref{c5.2}) read
$$
A\,||f||_{L^2(\R)}^2\leq\<Sf,f\>\leq B\,||f||_{L^2(\R)}^2\,,\quad (f\in L^2(\R))\,,  
$$
i.e., $AI_{L^2(\R)}\leq S\leq BI_{L^2(\R)}$.
In particular, the first inequality implies that $A\,||f||_{L^2(\R)}\leq ||Sf||_{L^2(\R)}$, for every $f\in L^2(\R)$. Since $S$ is positive, this fact is equivalent to the existence of the two-sided bounded inverse $S^{-1}$ of $S$ (see \cite[th.12.12.c]{Raf73}). That $S^{-1}$ exists implies that $\text{Range}(S)=\text{Range}(T_X)=L^2(\R)$. And a Bessel system $X$ is a frame if and only if this last condition is satisfied (see \cite[th.5.5.1]{Chr}).
Furthermore, $AI_{L^2(\R)}\leq S\leq BI_{L^2(\R)}$ implies that $0\leq I_{L^2(\R)}-B^{-1}S\leq \frac{B-A}{B}I_{L^2(\R)}$ and, consequently,
$$
||I_{L^2(\R)}-B^{-1}S||=\sup_{||f||=1}\big|\<(I_{L^2(\R)}-B^{-1}S)f,f\>\big|\leq \frac{B-A}{B}<1,
$$
so that
$$
S^{-1}=B^{-1}\sum_{k=0}^\infty (I_{L^2(\R)}-B^{-1}S)^k\,,
$$
where the last series converges in norm (uniformly). Since the set of decomposable operators is a $C^*$-algebra (moreover, a von Neumann algebra, see proposition \ref{prop1vn}), $S^{-1}$ is also a decomposable operator in the dilation representation given in proposition \ref{pdft} and (see \cite[prop.14.1.8]{KRII97})
$$
\G S^{-1}\G^{-1}=S(\om)^{-1}\,.
$$ 
(Note that $S(\om)$ and $S(\om)^{-1}$ are defined for a.e. $\om\in\partial\D$.)
Now, recall that for a bounded normal operator $S$ on a Hilbert space $\H$, $S$ has a bounded two-sided inverse if and only if $0<\beta:=\inf\{||Sx||_\H:x\in\H,\, ||x||_\H =1\}$ and, then, $||S^{-1}||=\beta^{-1}$; see \cite[lemma 2.4.8]{KRI97} and \cite[th.12.12.c]{Raf73}. Thus, being $S$, $\G S\G^{-1}$ and $S(\om)$ positive, the bounded two-sided inverse $S^{-1}$ exists if and only if 
\begin{eqnarray*}
\infty &>& \ds\frac{1}{\ds\einf_{\om\in\partial\D}\inf_{||u||_{l^2(\J)\oplus l^2(\J)}=1} ||S(\om)u||_{l^2(\J)\oplus l^2(\J)}}=
\esup_{\om\in\partial\D} \frac{1}{\ds\inf_{||u||_{l^2(\J)\oplus l^2(\J)}=1} ||S(\om)u||_{l^2(\J)\oplus l^2(\J)}}=
\\[2ex]
&=&\ds\esup_{\om\in\partial\D} ||S(\om)^{-1}||=||(\G S\G^{-1})^{-1}||=
\frac{1}{\ds\inf_{||\bu||_{L^2(\partial\D,l^2(\J)\oplus l^2(\J))}=1} ||(\G S\G^{-1})\bu||_{L^2(\partial\D,l^2(\J)\oplus l^2(\J))}}\,;
\end{eqnarray*}
see \cite[prop.14.1.8--9]{KRII97} for details.

3. If $X$ is a frame, in terms of {\it pseudo-inverses}\footnote{
Let $\H$, $\H'$ be Hilbert spaces and suppose that $U:\H\to\H'$ is a bounded operator with closed range $\RR_U$ and kernel $\NN_U$. The {\it pseudo-inverse} of $U$ is the unique operator $U|^{-1}:\H'\to\H$  satisfying
$\NN_{U|^{-1}}= \RR_U^\perp$, $\RR_{U|^{-1}}= \NN_U^\perp$ and
$UU|^{-1} f=f$ for $f\in\RR_U$ and $U|^{-1}U f=f$ for $f\in\NN_U^\perp$, where $\perp$ denotes orthogonal complement.
}, 
$$
S^{-1}=T_X^*|^{-1}T_X|^{-1}=(T_X|^{-1})^*T_X|^{-1}\,,
$$
$||T_X^*|^{-1}||=||T_X|^{-1}||$ and the frame bounds are $B=||T_X||^2=||T_X^*||^2$ and
$A=1/||T_X|^{-1}||^2=1/||T_X^*|^{-1}||^2$; see \cite[lemmas 5.5.4 and A.7.2]{Chr} for details. 
Thus, $X$ is a tight frame if and only if  $X$ is a frame and $||T_X||\cdot||T_X|^{-1}||=1$. In such case, for $f\in\NN_{T_X}^\perp$, 
$||f||=||T_X|^{-1}T_Xf||\leq ||T_X|^{-1}||\,||T_Xf||$, so that
$$
||T_X||\,||f||=\frac{1}{||T_X|^{-1}||}||f||\leq \,||T_Xf||\leq ||T_X||\,||f||
$$
and, then, $||T_X||\,||f||=||T_Xf||$. That is,  $T_X/||T_X||$ is a partial isometry with initial space $\NN_{T_X}^\perp$ and final space $L^2(\R)$. The orthogonal projection over the final space is
$$
I_{L^2(\R)}=\frac{T_XT_X^*}{||T_X||^2}=\frac{S}{||S||}=B^{-1} S\,.
$$
The converse follows from the fact that (\ref{c5.2}) is equivalent to $AI\leq S\leq BI$. Finally, that $\G U\G^{-1}$ is the constant diagonalizable operator $I_{l^2(\J)\oplus l^2(\J)}$ if and only if $U=I_{L^2(\R)}$ is just \cite[prop.14.1.8.iv]{KRII97}.
\end{proof}

\begin{remark}\label{rm7}\rm
In theorem \ref{th1}, the frame operator $S=T_XT_X^*$ has a decomposable image $\G S\G^{-1}$ on the dilation representation given in proposition \ref{pdft} thanks to the fact that $S$ and $D$ commute (see proposition \ref{prop1vn}). 
What about the commutation relations between $S$ and the translation (shift) operator $T$ in order that $S$ have decomposable images $\F S\F^{-1}$ and $\F_* S\F_*^{-1}$ on the translation representations given in propositions \ref{ptft} and \ref{ptft1}?
This question is the cornerstone to develop {\it fiberization techniques} for wavelet systems in $L^2(\R)$ of the form (\ref{ws}) on the Fourier domain. 

For the translation and dilation operators, $T$ and $D$,  defined on $L^2(\R)$ by (\ref{tdo}), one has $TD=DT^2$. Taking adjoints, $D^{-1}T^{-1}=T^{-2}D^{-1}$.
Also, $D=T^{-1}DT^2$ or $DT^{-2}=T^{-1}D$, and $D^{-1}=T^{-2}D^{-1}T$ or $T^2D^{-1}=D^{-1}T$. Thus, in general,
$$
\begin{array}{ll}
T^jD^k=D^kT^{j2^k}\,, & \text{ if } k>0 \text{ and } j\in\Z\,,
\\[1ex]
T^{j2^{|k|}}D^k=D^kT^{j}\,, & \text{ if } k<0 \text{ and } j\in\Z\,.
\end{array}
$$

Given a Bessel wavelet system $X$ in $L^2(\R)$ of the form (\ref{ws}) and the corresponding operator $S=T_XT_X^*$, for each $f\in L^2(\R)$,
$$
\begin{array}{rl}
\ds ST f&\ds =\sum_{\stackrel{k,j\in\Z}{\psi\in\Psi}}\< f,T^{-1}D^kT^j\psi\>D^kT^j\psi=
\\[4ex]
&\ds = \sum_{k\geq0}\sum_{\stackrel{j\in\Z}{\psi\in\Psi}}\< f,T^{-1}D^kT^{j}\psi\>D^kT^j\psi+ \sum_{k<0}\sum_{\stackrel{j\in\Z}{\psi\in\Psi}}\< f,T^{-1}D^kT^j\psi\>D^kT^j\psi=
\\[4ex]
&\ds =\sum_{k\geq0}\sum_{\stackrel{j\in\Z}{\psi\in\Psi}}\< f,D^kT^{j-2^k}\psi\>D^kT^{j}\psi+ \sum_{k<0}\sum_{\stackrel{j\in\Z}{\psi\in\Psi}}\< f,T^{j2^{|k|}-1}D^k\psi\>T^{j2^{|k|}}D^k\psi
\end{array}
$$
and 
$$
\begin{array}{rl}
\ds TS f&\ds =\sum_{\stackrel{k,j\in\Z}{\psi\in\Psi}}\< f,D^kT^j\psi\>TD^kT^j\psi=
\\[4ex]
&\ds = \sum_{k\geq0}\sum_{\stackrel{j\in\Z}{\psi\in\Psi}}\< f,D^kT^j\psi\>TD^kT^j\psi+ \sum_{k<0}\sum_{\stackrel{j\in\Z}{\psi\in\Psi}}\< f,D^kT^j\psi\>TD^kT^j\psi=
\\[4ex]
&\ds=\sum_{k\geq0}\sum_{\stackrel{j\in\Z}{\psi\in\Psi}}\< f,D^kT^j\psi\>D^kT^{j+2^k}\psi+ \sum_{k<0}\sum_{\stackrel{j\in\Z}{\psi\in\Psi}}\< f,T^{j2^{|k|}}D^k\psi\>T^{1+j2^{|k|}}D^k\psi\,.
\end{array}
$$
The sums for $k\geq 0$ coincide, but not the sums for $k<0$. If for each $k<0$ we add to the affine system $X$ the functions
$$
\psi_{k,j}^{l}=T^{2^{|k|}j+l}D^k\psi=T^lD^kT^j\psi=2^{k/2}\psi(2^k(\cdot-l)-j),\quad l=1,2,\ldots,2^{-k}-1,
$$
one obtains a system $\tilde X^q$ associated with $X$ such that the corresponding frame operator $S=T_{\tilde X^q}T_{\tilde X^q}^*$ commutes with the translation operator $T$. Thus, such $S$ shall be a decomposable operator on any spectral representation of $T$. Moreover, we have $\tilde X^q = \tilde X^q_+ \cup \tilde X^q_-$, where
$$
\begin{array}{l}
\tilde X^q_+ = \big\{\psi_{k,j}:=D^k T^j\psi: \psi\in\Psi,\,k\geq0,\,j\in\Z\big\} = \big\{T^a D^k T^b\psi: \psi\in\Psi,\,k\geq0,\,a\in\Z,\,0\leq b <2^k\big\}\,,\\[2ex]
\tilde X^q_- = \big\{\psi_{k,j}^{l}:=T^l D^k T^j\psi: \psi\in\Psi,\,k<0,\,j\in\Z,\,0\leq l<2^{-k}\big\} = \big\{T^a D^k\psi: \psi\in\Psi,\,k<0,\,a\in\Z\big\}\,.
\end{array}
$$
A variant of $\tilde X^q$ is what Ron and Shen \cite[Section 5]{RS97a} call the {\it quasi-affine system} $X^q$ associated with $X$. $X^q = X^q_+ \cup X^q_-$, where $X^q_+=\tilde X^q_+$, the truncated affine system $X_0$ according to Ron and Shen \cite[Section 4]{RS97a}, and
$$
\begin{array}{rl}
X^q_- &= \big\{2^{k/2}\psi_{k,j}^{l}:=2^{k/2}T^l D^k T^j\psi: \psi\in\Psi,\,k<0,\,j\in\Z,\,0\leq l<2^{-k}\big\} = 
\\[1ex]
&= \big\{2^{k/2}T^a D^k\psi: \psi\in\Psi,\,k<0,\,a\in\Z\big\}\,.
\end{array}
$$
Working in the Fourier domain, one is forced to consider the translation invariant system $X^q$ or $\tilde X^q$.  Ron and Shen \cite[Theorem 5.5]{RS97a} prove a variant of the following result: The wavelet system $X$ is a frame if and only if its quasi-affine counterparts $X^q$ or $\tilde X^q$ are a frame.  In particular, the frame $X$ is tight if and only if the quasi-affine system $X^q$ or $\tilde X^q$ is tight. Furthermore, the two systems $X$ and $X^q$ have identical frame bounds.
The choice of the dilation representation of proposition \ref{pdft} (or any other spectral representation for $D$) avoids this inconvenience, since a wavelet system (in general, any affine system) is dilation invariant.
\end{remark}

For a Bessel wavelet system $X$ in $L^2(\R)$ of the form (\ref{ws}), the operator $S=T_XT_X^*$ in the {dilation representation} of proposition \ref{pdft}, $\G S\G^{-1}$, is given by
$$
L^2\big(\partial {\D};l^2(\J)\oplus l^2(\J)\big)
\stackrel{\G^{-1}}{\longrightarrow} L^2(\R)\stackrel{S}{\longrightarrow}L^2(\R)\stackrel{\G}{\longrightarrow}L^2\big(\partial {\D};l^2(\J)\oplus l^2(\J)\big)
$$
$$
\G f=\tilde{\mathbf f}=\big\{\tilde f_{s,j}^{(m)}\big\}\mapsto f\mapsto Sf={\sum_{\stackrel{k,j\in\Z}{\psi\in\Psi}}}^u\< f,\psi_{k,j}\>_{L^2(\R)}\,\psi_{k,j}\mapsto {\sum_{\stackrel{k,j\in\Z}{\psi\in\Psi}}}^u\< f,\psi_{k,j}\>_{L^2(\R)}\,\G\psi_{k,j}\,.
$$
The superindex 'u' added to the sum symbol $\sum$ in the last expressions reflects that the series defining $S$ converges unconditionally for all $f\in L^2(\R)$; see \cite[Corollary 3.2.5]{Chr}. 

\begin{lemma}\label{lcu}
{\rm \cite[Lemma 2.1.1]{Chr}}
 Let $\{y_k\}_{k=1}^\infty$ be a sequence in a Banach space Y, and let
$y\in Y$. Then the following are equivalent:

(i) $ \sum_{k=1}^\infty y_k$ converges unconditionally to $y$ in $Y$.

(ii) For every $\epsilon>0$ there exists a finite set $F$ such that
$||y-\sum_{k\in I} y_k||\leq\epsilon$ for all finite sets $I\subset\N$ containing $F$.
\end{lemma}
According to lemma \ref{lcu}, $\left[\sum_{\stackrel{k,j\in\Z}{\psi\in\Psi}}\right]^u$ means that, for each $f\in L^2(\R)$, one must take the limit of sums over suitable finite sets of triplets $(k,j,\psi)\in \Z\times\Z\times\Psi$.
This is the correct way to interpret the sums and avoids any possible ``infinity'' in partial calculations dealing with the expressions we will encounter in what follows. 

\medskip

In order to take advantage of the {dilation representation} of proposition \ref{pdft}, the matrix $\big(\alpha_{i,n}^{s,j,m}\big)$, defined by (\ref{chm}), must appear on stage. 
Next result gives a expression for the matrix elements and fibers of the decomposable operator $\G S\G^{-1}$ associated with the Bessel wavelet system $X$. They are written in terms of the $\alpha_{i,n}^{s,j,m}$'s and the components $\big\{\hat\psi_i^{(n)}\big\}$ of each $\psi\in\Psi$ (and not in terms of the components $\big\{\tilde\psi_{s,j}^{(m)}\big\}$!). 
The result is given for the ONB $\big\{u_{s,l}\big\}_{l\in\J,s=\pm}$ of $l^2(\J)\oplus l^2(\J)$ fixed in proposition \ref{pdft}.

\begin{theor}\label{th2}
For a Bessel wavelet system $X$ in $L^2(\R)$ of the form (\ref{ws}), the operator $S=T_XT_X^*$ in the {dilation representation} of proposition \ref{pdft}, $\G S\G^{-1}$, has matrix elements $[\G S\G^{-1}]_{s,l}^{s',l'}$ given by
$$
\begin{array}{rl}
[\G S\G^{-1}]_{s,l}^{s',l'}:&L^2(\partial\D,\C)\longrightarrow L^2(\partial\D,\C)
\\[1ex]
&\ds h(\om)\mapsto h(\om)\,\sum_{\sigma}\om^{\sigma}
\sum_{\stackrel{i,n}{i',n'}}
\big({\sum_{k,j\in\Z}}^u\, \overline{\alpha_{i,n+j}^{s,l,k}}\,\alpha_{i',n'+j}^{s',l',k+\sigma}\big)\big({\sum_{\psi\in\Psi}}^u\, \overline{\hat\psi_{i}^{(n)}}\,\hat\psi_{i'}^{(n')}\big)\,,
\end{array}
$$
where $l,l'\in\J$, $s,s'=\pm$.
Thus, the fibers of $\G S\G^{-1}=S(\om)$ are
$$
\begin{array}{rl}
S(\om):&l^2(\J)\oplus l^2(\J)\longrightarrow l^2(\J)\oplus l^2(\J)
\\[1ex]
&\ds u_{s,l}\mapsto \bigoplus_{s',l'} u_{s',l'}\,\sum_{\sigma}\om^{\sigma}
\sum_{\stackrel{i,n}{i',n'}}
\big({\sum_{k,j\in\Z}}^u\, \overline{\alpha_{i,n+j}^{s,l,k}}\,\alpha_{i',n'+j}^{s',l',k+\sigma}\big)\big({\sum_{\psi\in\Psi}}^u\, \overline{\hat\psi_{i}^{(n)}}\,\hat\psi_{i'}^{(n')}\big)\,,
\end{array}
$$
for a.e. $\om\in\partial\D$.
\end{theor}

\begin{proof}
The following identities are direct consequences of 
(\ref{fdobl2}), (\ref{fdobl1}) and (\ref{chm}):
\be\label{ee1}
\begin{array}{rl}
\ds \sum_{r,p,q}\overline{\tilde\psi_{r,p}^{(q)}}\,\alpha_{i,n-j}^{r,p,q}
&\ds = \sum_{r,p,q}\overline{\tilde\psi_{r,p}^{(q)}}\,\<L_i^{(n-j)},K_{r,p}^{(q)}\>_{L^2(\R)}=
\\[2ex]
&\ds = \<L_i^{(n-j)},\sum_{r,p,q}\tilde\psi_{r,p}^{(q)}\,K_{r,p}^{(q)}\>_{L^2(\R)}=\<L_i^{(n-j)},\psi\>_{L^2(\R)}=\overline{\hat\psi_i^{(n-j)}}.
\end{array}
\ee
For $f=K_{s,l}^{(m)}$ one has
$$
\om^m u_{s,l}\stackrel{\G^{-1}}{\mapsto} K_{s,l}^{(m)}\stackrel{S}{\mapsto} SK_{s,l}^{(m)}={\sum_{\stackrel{k,j\in\Z}{\psi\in\Psi}}}^u\< K_{s,l}^{(m)},\psi_{k,j}\>_{L^2(\R)}\,\psi_{k,j}\stackrel{\G}{\mapsto} {\sum_{\stackrel{k,j\in\Z}{\psi\in\Psi}}}^u\<K_{s,l}^{(m)},\psi_{k,j}\>_{L^2(\R)}\,\G\psi_{k,j}.
$$
Using (\ref{fdobl1}) and the definition of $\G$ in proposition \ref{pdft},
$$
\begin{array}{rl}
\ds{\sum_{\stackrel{k,j\in\Z}{\psi\in\Psi}}}^u\<K_{s,l}^{(m)},\psi_{k,j}\>\G\psi_{k,j}
&\ds = {\sum_{\stackrel{k,j\in\Z}{\psi\in\Psi}}}^u\overline{\widetilde{[D^kT^j\psi]}_{s,l}^{(m)}}\,\big(\sum_{m',s',l'}\om^{m'}\,\widetilde{[D^kT^j\psi]}_{s',l'}^{(m')}\,u_{s',l'}\big)=
\\[2ex]
&\ds = \sum_{s',l'}\big(\om^m\,{\sum_{\stackrel{k,j\in\Z}{\psi\in\Psi}}}^u\sum_{m'}\om^{m'-m}\overline{\widetilde{[T^j\psi]}_{s,l}^{(m-k)}}\,\widetilde{[T^j\psi]}_{s',l'}^{(m'-k)}\big)\,u_{s',l'}=
\\[2ex]
&\ds = \sum_{s',l'}\big(\om^m\,{\sum_{\stackrel{k,j\in\Z}{\psi\in\Psi}}}^u\sum_{\sigma}\om^{\sigma}\overline{\widetilde{[T^j\psi]}_{s,l}^{(m-k)}}\,\widetilde{[T^j\psi]}_{s',l'}^{(m-k+\sigma)}\big)\,u_{s',l'}\,.
\end{array}
$$
By \cite[Lemma 5]{GS11a}, the last expression coincides with 
$$
\sum_{s',l'}\Big(\om^m\,{\sum_{\stackrel{k,j\in\Z}{\psi\in\Psi}}}^u
\sum_{\sigma}\om^{\sigma}
\big(\sum_{i,n} \overline{\alpha_{i,n}^{s,l,m-k}}\,\sum_{r,p,q} \alpha_{i,n-j}^{r,p,q}\overline{\tilde\psi_{r,p}^{(q)}}\big)
\big(\sum_{i',n'} \alpha_{i',n'}^{s',l',m-k+\sigma}\,\sum_{r',p',q'} \overline{\alpha_{i',n'-j}^{r',p',q'}}\tilde\psi_{r',p'}^{(q')}\big)\Big)\,u_{s',l'}
$$
and, by (\ref{ee1}), this is equal to
$$
\begin{array}{l}
\ds\sum_{s',l'}\Big(\om^m\,{\sum_{\stackrel{k,j\in\Z}{\psi\in\Psi}}}^u
\sum_{\sigma}\om^{\sigma}
\big(\sum_{i,n} \overline{\alpha_{i,n}^{s,l,m-k}}\,\overline{\hat\psi_{i}^{(n-j)}}\big)
\big(\sum_{i',n'} \alpha_{i',n'}^{s',l',m-k+\sigma}\,\hat\psi_{i'}^{(n'-j)}\big)\Big)\,u_{s',l'}=
\\[2ex]
\ds = \sum_{s',l'}\Big(\om^m\,\sum_{\sigma}\om^{\sigma}
\sum_{\stackrel{i,n}{i',n'}}
\big({\sum_{k,j\in\Z}}^u\, \overline{\alpha_{i,n+j}^{s,l,k}}\,\alpha_{i',n'+j}^{s',l',k+\sigma}\big)\big({\sum_{\psi\in\Psi}}^u\, \overline{\hat\psi_{i}^{(n)}}\,\hat\psi_{i'}^{(n')}\big)\Big)\,u_{s',l'}\,.
\end{array}
$$
Thus, the matrix element $[\G S\G^{-1}]_{s,l}^{s',l'}$ satisfies
$$
[\G S\G^{-1}]_{s,l}^{s',l'}(\om^m)=\om^m\,\sum_{\sigma}\om^{\sigma}
\sum_{\stackrel{i,n}{i',n'}}
\big({\sum_{k,j\in\Z}}^u\, \overline{\alpha_{i,n+j}^{s,l,k}}\,\alpha_{i',n'+j}^{s',l',k+\sigma}\big)\big({\sum_{\psi\in\Psi}}^u\, \overline{\hat\psi_{i}^{(n)}}\,\hat\psi_{i'}^{(n')}\big)\,,\quad (m\in\Z)\,.
$$
Since $\{\om^m\}_{m\in\Z}$ is an ONB of $L^2(\partial\D,\C)$, we get the result.
\end{proof}

In particular, for tight wavelet frames, 

\begin{coro}\label{coro10}
Let $X$ be a wavelet system in $L^2(\R)$ of the form (\ref{ws}) and such that (\ref{aePsi}) is satisfied.. 
Then, $X$ is a {tight frame} for $L^2(\R)$, with frame bound $B$, if and only if
\be\label{sesg}
\sum_{\stackrel{i,n}{i',n'}}
\big({\sum_{k,j\in\Z}}^u\, \overline{\alpha_{i,n+j}^{s,l,k}}\,\alpha_{i',n'+j}^{s',l',k+\sigma}\big)\big({\sum_{\psi\in\Psi}}^u\, \overline{\hat\psi_{i}^{(n)}}\,\hat\psi_{i'}^{(n')}\big)=
B\,\delta_{s,s'}\delta_{l-l'}\delta_\sigma,
\ee
where $s,s'=\pm$, $l,l'\in\J$, $\sigma\in\Z$, and $\delta$ denotes the Dirac $\delta$-function. 
\end{coro}

\begin{proof}
By theorem \ref{th1}, $X$ is a {tight frame} with frame bound $B$ if and only if $S=B\,I_{L^2(\R)}$ or, equivalently, $\G S\G^{-1}=B\, I_{l^2(\J)\oplus l^2(\J)}$, i.e., 
$S(\om)=B\, I_{l^2(\J)\oplus l^2(\J)}$, for a.e. $\om\in\partial\D$. Now, consider the expression of the fibers $S(\om)$ given in theorem \ref{th2}.
\end{proof}


\section{Extension principles for wavelet frames}\label{s5}

In the theory of wavelet frames one of the main trends of development is based on the so-called {\it extension principles} and {\it multiresolution analysis} (MRA). We include here a brief translation of this approach to the framework of the spectral techniques. The main tool for such translation is the {\it periodized Fourier transform} given in proposition \ref{ptft1}. 
As before, we restrict attention to the univariate case and dyadic dilation.

\begin{remark}\label{rm38} 
\rm
Note that the periodized Fourier transform of proposition \ref{ptft1} leads to the following notation: For $f\in L^2(\R)$,
$$
\begin{array}{l}
\hat f\text{ denotes the usual Fourier transform of $f$ defined in (\ref{FT})},
\\
\hat f_*=\F_*f \text{ denotes the periodized Fourier transform of $f$ given by (\ref{tft1})},
\\
\hat f_k\text{ is the $k$-component ($k\in\Z$) of the periodized Fourier transform $\hat f_*$ of $f$, see (\ref{tft1}) and (\ref{tft1a})}.
\end{array}
$$
\end{remark}

\begin{remark}\label{rm39} 
\rm
The {\it bracket product} defined by
\be\label{bp}
[f, g](\theta):=\sum_{k\in\Z} f(\theta+k)\,\overline{g(\theta+k)}\,,\quad \text{ for a.e. }\theta\in\R\text{ and }f,g\in L^2(\R)\,,
\ee
plays a key role in the theory of shift-invariant systems \cite{BVR94a}, mainly used in the Fourier domain. The bracket product has a clear meaning under the periodized Fourier transform. 
Indeed, going from $\theta\in\R$ to $\om=e^{2\pi i\theta}\in\partial\D$, one has
$$
[\hat f, \hat g](\theta)=\<\hat f_*(\omega),\hat g_*(\omega)\>_{l^2(\Z)}\,,\quad 
\text{ for a.e. }\theta\in\R\text{ and }f,g\in L^2(\R)\,.
$$
In particular,
$$
[\hat f, \hat f](\theta)=||\hat f_*(\omega)||_{l^2(\Z)}\,,\quad 
\text{ for a.e. }\theta\in\R\text{ and }f\in L^2(\R)\,.
$$
In what follows, given $f\in L^2(\R)$, we shall write
$$
\sigma(f):=\text{supp}[\hat f,\hat f]=\{\theta\in\R:[\hat f,\hat f](\theta)\neq0\}
$$
or
$$
\sigma(f):=\text{supp}||\hat f_*||_{l^2(\Z)}=\{\om\in\partial\D:||\hat f_*(\om)||_{l^2(\Z)}\neq0\}
$$
depending on the use of $\hat f$ or $\hat f_*$, respectively.
\end{remark}

Now, consider the principal shift-invariant subspace generated by a function $\phi\in L^2(\R)$, i.e., the $L^2(\R)$-closure of the subspace generated by the integer translates of $\phi$:
$$
V_0:=\overline{\text{span}}\{T^k\phi:k\in\Z\}\,.
$$
Let $P_0$ denote the orthogonal projection of $L^2(\R)$ onto $V_0$.

The following result characterizes the subspace $V_0$ and gives an expression for $P_0$ in the spectral model of proposition \ref{ptft1}.

\begin{prop}\label{psis}
Let $\phi\in L^2(\R)$, consider the subspace $V_0=\overline{\text{span}}\{T^k\phi:k\in\Z\}$ and let $P_0$ be the orthogonal projection of $L^2(\R)$ onto $V_0$.
Then, $f\in L^2(\R)$ belongs to $V_0$ if and only if  $\hat f_*(\om)$ is colineal to $\hat \phi_*(\om)$ for a.e. $\om\in\partial\D$. Thus, for $f\in L^2(\R)$,
\be\label{ep0}
[\widehat{P_0 f}]_*(\om)=
\left\{\begin{array}{ll}
\ds\frac{\<\hat f_*(\om),\hat\phi_*(\om)\>_{l^2(\Z)}}{||\hat\phi_*(\om)||^2_{l^2(\Z)}}\,\hat\phi_*(\om)\,,& \text{ for a.e. }\om\in\sigma(\phi)\,,
\\
0\,,& \text{ for a.e. }\om\notin\sigma(\phi)\,.
\end{array}\right.
\ee
\end{prop}

\begin{proof}
Due to (\ref{ditf}), one has $\F_*V_0=\overline{\text{span}}\{\om^k\cdot\hat\phi_*(\om):k\in\Z\}$.
Since $\F_*$ is unitary, $[\F_*V_0]^\perp=\F_*[V_0^\perp]$. Thus, $f\in V_0^\perp$ if and only if
$$
\<\hat f_*,\om^k\hat\phi_*\>_{L^2(\partial\D,l^2(\Z))}=\int_{\partial\D}\om^{-k}\<f_*(\om),\hat\phi_*(\om)\>_{l^2(\Z)}\,\frac{d\om}{2\pi}=0\,,\quad (k\in\Z)\,.
$$
The last condition is equivalent to 
$$
\<f_*(\om),\hat\phi_*(\om)\>_{l^2(\Z)}=0\,,\quad \text{ for a.e }\om\in\partial {\D}\,.
$$
Thus, for each $f\in L^2(\R)$, $f\in V_0$ if and only if $\hat f_*(\om)$ is colineal to $\hat \phi_*(\om)$ for a.e. $\om\in\partial\D$. Then, for each fixed $\om\in\partial\D$, one gets (\ref{ep0}) for the orthogonal projection of $l^2(\Z)$ onto the $1$-dimensional subspace generated by $\hat\phi_*(\om)$ .
\end{proof}

Rewriting proposition \ref{psis} in terms of the usual Fourier transform defined in (\ref{FT}) and the bracket product given in remark \ref{rm39} one obtains the following two classical results in the theory of shift-invariant subspaces --see, for example, Theorems 2.9 and 2.14 in \cite{BVR94b}--.

\begin{coro}\label{bvr2.9}
For each $f\in L^2(\R)$, $\widehat{P_0 f}=H^f\hat\phi$, where the $1$-periodic function $H^f$ is defined by
\be\label{ep1}
H^f(\theta)=\left\{\begin{array}{ll} 
\ds \frac{[\hat f,\hat\phi](\theta)}{[\hat \phi,\hat\phi](\theta)}\,,& \text{ for a.e. }\theta\in\sigma(\phi)\,, \\ 0\,,& \text{ for a.e. }\theta\notin\sigma(\phi)\,.\end{array}\right.
\ee
\end{coro}

\begin{proof}
Looking at (\ref{tft1}) and (\ref{tft1a}), and since $\{u_k\}_{k\in\Z}$ is an orthonormal basis of $l^2(\Z)$, for $\om=e^{2\pi i\theta}$,
$$
\begin{array}{rl}
\ds\<\hat f_*(\om),\hat\phi_*(\om)\>_{l^2(\Z)}\,\hat\phi_*(\om) &
\ds =\bigoplus_{k\in\Z}\Big[\sum_{j\in\Z} \hat f_j(\om)\,\overline{\hat\phi_j(\om)}\Big]\,\hat\phi_k(\om)\,u_k\,,\quad \text{ for a.e }\om\in\partial {\D}\,,
\\
& \ds =\Big[\sum_{j\in\Z} \overline{\hat f(\theta+j)}\,\hat\phi(\theta+j)\Big]\,\bigoplus_{k\in\Z}\overline{\hat\phi(\theta+k)}\,u_k\,,\quad \text{ for a.e }\theta\in\R\,,
\end{array}
$$
$$
\begin{array}{rl}
\ds ||\hat\phi_*(\om)||_{l^2(\Z)} &
\ds =\sum_{j\in\Z} |\hat\phi_j(\om)|^2\,,\quad \text{ for a.e }\om\in\partial {\D}\,,
\\
& \ds = \sum_{j\in\Z} |\hat\phi(\theta+j)|^2\,,\quad \text{ for a.e }\theta\in\R\,,
\end{array}
$$
and
$$
\begin{array}{rl}
\ds [\widehat{P_\phi f}]_*(\om)
& \ds =\bigoplus_{k\in\Z} [\widehat{P_\phi f}]_k(\om)\,u_k\,,\quad \text{ for a.e }\om\in\partial {\D}\,,
\\
& \ds =\bigoplus_{k\in\Z} \overline{[\widehat{P_\phi f}](\theta+k)}\,u_k\,,\quad \text{ for a.e }\theta\in\R\,,
\end{array}
$$
Substituting these expressions in (\ref{ep0}) and using (\ref{bp}), one gets the result.
\end{proof}

\begin{coro}\label{bvr2.14}
Let $\phi\in L^2(\R)$ and consider the subspace $V_0:=\overline{\text{span}}\{T^k\phi:k\in\Z\}$. A function $f\in L^2(\R)$ is in $V_0$ if and only if $\hat f_*=H_*^f\hat\phi_*$ for some measurable function $H_*^f:\partial\D\to\C$ with $H_*^f\hat\phi_*\in L^2(\partial\D,l^2(\Z))$.
Equivalently, a function $f\in L^2(\R)$ is in $V_0$ if and only if $\hat f=H^f\hat\phi$ for some $1$-periodic measurable function $H^f:\R\to\C$ with $\tilde H^f\hat\phi\in L^2(\R)$.
For a function $f\in V_0$, both functions $H_*^f$ and $H_f$ are related by
\be\label{hfhft}
H_*^f(\om)=\overline{H^f(\theta)}\,,\quad \text{for a.e. }\om=e^{2\pi i\theta}\,. 
\ee
\end{coro}

\begin{proof}
The result is a direct consequence of Proposition \ref{psis} and Corollary \ref{bvr2.9}. Moreover, by (\ref{tft1a}), for a.e. $\om=e^{2\pi i\theta}$, 
$\hat f_k(\om)=H_*^f(\om)\,\hat\phi_k(\om)$ if and only if $\overline{\hat f(\theta+k)}=H^f(e^{2\pi i\theta})\overline{\hat\phi(\theta+k)}$, for every $k\in\Z$; this implies (\ref{hfhft}).
\end{proof}

In order to accommodate the discussion to the cited references, we interchange the roles of the indices $j$ and $k$ used until now, that is, given a subset $\Psi$ of $L^2(\R)$, the wavelet system $X=X(\Psi)$ we consider from now on is not of the form (\ref{ws}), else of the form
\be\label{wsxx}
X(\Psi):=\big\{D^jT^k\psi: \psi\in\Psi,\,j,k\in\Z\big\}\,.
\ee
We apologize for the inconvenience.

In principle, the study focuses on wavelet systems $X(\Psi)$ of the form (\ref{wsxx}) that are derived from a {\it refinable function} $\phi\in L^2(\R)$.  Here refinability of $\phi$ means $D^{-1}\phi\in V_0$. By Corollary \ref{bvr2.14}, this is equivalent to the existence of a $1$-periodic measurable function $H^{D^{-1}\phi}$ on $\R$ (or measurable $H_*^{D^{-1}\phi}$ on $\partial\D$), the {\it refinement mask}, such that $\widehat{[D^{-1}\phi]}(\theta)=H^{D^{-1}\phi}(\theta)\hat\phi(\theta)$ (or $\widehat{[D^{-1}\phi]}_*(\om)=H_*^{D^{-1}\phi}(\om)\hat\phi_*(\om)$). According to Proposition \ref{bvr2.9},
\be\label{rmask}
\widehat{[D^{-1}\phi]}(\theta)=H^{D^{-1}\phi}(\theta)\hat\phi(\theta)
=\left\{\begin{array}{ll} 
\ds \frac{[\widehat{[D^{-1}\phi]},\hat\phi](\theta)}{[\hat \phi,\hat\phi](\theta)}\hat\phi(\theta)\,,& \text{ for a.e. }\theta\in\sigma(\phi)\,, \\ 0\,,& \text{ for a.e. }\theta\notin\sigma(\phi)\,.\end{array}\right.\,.
\ee
Or, by Proposition \ref{psis},
\be\label{rmaskw}
\widehat{[D^{-1}\phi]_*}(\om)=H_*^{D^{-1}\phi}(\om)\hat\phi_*(\om)
=
\left\{\begin{array}{ll}
\ds\frac{\<\widehat{[D^{-1}\phi]}_*(\om),\hat\phi_*(\om)\>_{l^2(\Z)}}{||\hat\phi_*(\om)||^2_{l^2(\Z)}}\,\hat\phi_*(\om)\,,& \text{ for a.e. }\om\in\sigma(\phi)\,,
\\
0\,,& \text{ for a.e. }\om\notin\sigma(\phi)\,.
\end{array}\right.
\ee

The next results collect well-known facts, see e.g. \cite{BVR93,BeLi98}.

\begin{prop}\label{rfma}
Let $\phi\in L^2(\R)$ be a refinable function, $V_0=\overline{\text{span}}\{T^k\phi:k\in\Z\}$, and let $V_j=D^jV_0$, for $j\in\Z$. Then:
\begin{itemize}
\item[(a)]
$V_j\subseteq V_{j+1}$, for $j\in\Z$.
\item[(b)]
$\cap_j V_j=\{0\}$. 
\item[(c)]
$\overline{\cup_j V_j}=L^2(\R)$ if and only if  
$\cup_j 2^j[\text{supp}\,\hat\phi]=\R$ (modulo a null-set),\footnote{Recall that here the support of an $L^2(\R)$-function $f$ is defined only modulo a
null-set as $\text{supp}\,f:=\{\theta\in\R:f(\theta)\neq 0\}$.}
 where, given $S\subseteq\R$, $2^jS:=\{2^j\theta:\theta\in S\}$. 
\end{itemize}
\end{prop}

\begin{proof}
(a): Since $\phi$ is refinable, one has $D^{-1}\phi\in V_0$. 
Being $D^{-1}$ continuous, 
$$
V_{-1}=D^{-1}V_0=\overline{\text{span}}\{D^{-1}T^k\phi:k\in\Z\}=\overline{\text{span}}\{T^{2k}D^{-1}\phi:k\in\Z\}\,.
$$
These facts, together with the shift-invariance of $V_0$, imply $V_{-1}\subseteq V_0$. Then, $V_j=D^{j+1}V_{-1}\subseteq D^{j+1}V_{0}=V_{j+1}$, for $j\in\Z$. This result has been proved by Benedetto and Li \cite[Theorem 4.4]{BeLi98} when the refinement mask $H^{D^{-1}\phi}$ belongs to $L^\infty(\partial\D)$.

(b): This result is a particular case of \cite[Corollary 4.14]{BVR93}.

(c): Here the main property is the shift-invariance of  $\overline{\cup_j V_j}$, so that its Fourier transform must be of the form $L^2(\Omega)$ for some measurable set $\Omega\subseteq\R$. See \cite[Theorem 4.3]{BVR93} for details.
\end{proof}

\begin{coro}\label{ccws}
Let $\phi\in L^2(\R)$ be a refinable function, $V_0=\overline{\text{span}}\{T^k\phi:k\in\Z\}$, and let $V_j=D^jV_0$, for $j\in\Z$. If $\hat\phi$ is
nonzero a.e. in some neighbourhood of the origin, then $\overline{\cup_j V_j}=L^2(\R)$.
\end{coro}

\begin{proof}
Let $\Omega$ be a neighbourhood of the origin such that $\hat\phi$ is
nonzero a.e. in $\Omega$. Then, $\cup_j 2^j[\text{supp}\,\hat\phi]\subseteq \cup_j 2^j\Omega=\R$, and the result follows from proposition \ref{rfma}.c. 
\end{proof}

In general it is not required that $\phi$ be a ``good'' generator for $V_0$ in the sense that $\{T^k\phi:k\in\Z\}$ is a basis or a (pseudo-)frame for $V_0$.
When the nested sequence $\{V_j\}$ satisfies $\overline{\cup_{j\in\Z} V_j}=L^2(\R)$, one says that {\it  $\phi$ generates the (generalized) multiresolution analysis} (MRA) $\{V_j\}$ of $L^2(\R)$.

Now, let $\Psi=\{\psi^1,\ldots,\psi^s\}$ be a finite subset of $L^2(\R)$ and assume that $\Psi\subset V_1$. Obviously, this implies that there exist $1$-periodic measurable functions $H^{D^{-1}\psi^l}$ on $\R$ (or measurable $H_*^{D^{-1}\psi^l}$ on $\partial\D$), $l=1,\ldots,s$, the {\it wavelet masks}, such that
\be\label{rmaskx}
\widehat{[D^{-1}\psi^l]}(\theta)=H^{D^{-1}\psi^l}(\theta)\hat\phi(\theta)
=\left\{\begin{array}{ll} 
\ds \frac{[\widehat{[D^{-1}\psi^l]},\hat\phi](\theta)}{[\hat \phi,\hat\phi](\theta)}\hat\phi(\theta)\,,& \text{ for a.e. }\theta\in\sigma(\phi)\,, \\ 0\,,& \text{ for a.e. }\theta\notin\sigma(\phi)\,.\end{array}\right.\,.
\ee
\be\label{rmaskwx}
\widehat{[D^{-1}\psi^l]_*}(\om)=H_*^{D^{-1}\psi^l}(\om)\hat\phi_*(\om)
=
\left\{\begin{array}{ll}
\ds\frac{\<\widehat{[D^{-1}\psi^l]}_*(\om),\hat\phi_*(\om)\>_{l^2(\Z)}}{||\hat\phi_*(\om)||^2_{l^2(\Z)}}\,\hat\phi_*(\om)\,,& \text{ for a.e. }\om\in\sigma(\phi)\,,
\\
0\,,& \text{ for a.e. }\om\notin\sigma(\phi)\,.
\end{array}\right.
\ee

The {\it Unitary Extension Principle} (UEP), introduced by Ron and Shen \cite{RS97a,RS97b}, gives a sufficient condition for $X(\Psi)$ to be a tight frame. The condition is written in terms of the matrix functions
$$
H(\theta):=\left(
\begin{array}{cc} H^{D^{-1}\phi}(\theta) & H^{D^{-1}\phi}(\theta+1/2) \\ H^{D^{-1}\psi^1}(\theta) & H^{D^{-1}\psi^1}(\theta+1/2) \\ \vdots &\vdots \\ H^{D^{-1}\psi^s}(\theta) & H^{D^{-1}\psi^s}(\theta+1/2) \end{array}\right)
\quad\text{or}\quad
H_*(\om):=\left(
\begin{array}{cc} H^{D^{-1}\phi}_*(\om) & H^{D^{-1}\phi}_*(-\om) \\ H^{D^{-1}\psi^1}_*(\om) & H^{D^{-1}\psi^1}_*(-\om) \\ \vdots &\vdots \\ H^{D^{-1}\psi^s}_*(\om) & H^{D^{-1}\psi^s}_*(-\om) \end{array}\right)\,.
$$
The following result is a refined version of the UEP due to Benedetto and Trieber \cite[Theorem 1.7.1]{BT01}:

\begin{theor}[Unitary Extension Principle]\label{tuep}
Let $\phi\in L^2(\R)$ be a refinable function such that 
\be\label{ftpc0}
\lim_{\theta\to 0}\hat\phi(\theta)=1\,.   
\ee
Let $V_0=\overline{\text{span}}\{T^k\phi:k\in\Z\}$ and let $V_j=D^jV_0$, for $j\in\Z$. Consider a finite set $\Psi=\{\psi^1,\ldots,\psi^s\}\subset V_1$. If  
\be\label{ftpc0x}
H^*(\theta)H(\theta)=Id\quad\text{or}\quad  H_*^*(\om)H_*(\om)=Id\,,\quad \text{a.e. on }\sigma(\phi)\,,
\ee
then $X(\Psi)$ is a tight wavelet frame with frame bound $1$ for $L^2(\R)$.
\end{theor}
Condition (\ref{ftpc0}) is related with the completeness of the wavelet system through corollary \ref{ccws}. Due to (\ref{hfhft}), the two conditions in (\ref{ftpc0x}) coincide.

Restricting attention to compactly supported functions satisfying certain additional conditions, Chui and He \cite[Lemma 1]{CH00} give a complete characterization of what they call minimum-energy frames associated with a given refinable function, which is closely related to the UEP.

Fan, Heinecke and Shen \cite{FHS16} develop a {\it duality principle} based on the unitary equivalence of the frame operator and the Gramian of certain adjoint systems. When applied to fiberization techniques \cite{RS98,RS99}, the duality principle leads to simple methods of constructing dual wavelet frames. 
Since the adjoint system of a tight frame is an orthonormal sequence, the construction scheme reduces to complete a constant matrix so that its columns are pairwise orthogonal. Details about the connection between the unitary extension principle and the duality principle can be found in Fan, Ji and Shen \cite[Section 4]{FJS16}.

The UEP is subsequently extended by Daubechies, Han, Ron and Shen \cite{DRS03} and Chui, He and St\"ockler \cite{CHS02} in the form of the {\it Oblique Extension Principle} (OEP). The key idea is to consider different (equivalent) refinable functions the (homogeneous) wavelet system $X(\Psi)$ may be derived from. Perhaps the best analysis of the OEP is done by Han in \cite{H10} for dual pairs of nonhomogeneous wavelet systems in a distribution setting; see, in particular, \cite[Theorem 2, Theorem 9 and Corollary 10]{H10}. 

Given two subsets $\Phi$ and $\Psi$ of $L^2(\R)$ and an integer $J$, by a {\it nonhomogeneous wavelet system} $X_J(\Phi,\Psi)$ generated by $\Phi$ and $\Psi$ we mean a system of the form
$$
X_J(\Phi,\Psi):=\big\{D^{J}T^k\phi:\phi\in\Phi,\,k\in\Z\big\}\cup
\big\{D^{j}T^k\psi:\psi\in\Psi,\,k\geq J,\,k\in\Z\big\}\,.
$$
In particular, let us consider four finite sets of $L^2(\R)$:  
\be\label{fsphis}
\Phi=\{\phi^1,\ldots,\phi^r\}\,,\quad 
\Psi=\{\psi^1,\ldots,\psi^s\}
\quad\text{and}\quad
\tilde\Phi=\{\tilde\phi^1,\ldots,\tilde\phi^r\}\,,\quad 
\tilde\Psi=\{\tilde\psi^1,\ldots,\tilde\psi^s\}\,.
\ee
The pair $X_J(\Phi,\Psi)$ and $X_J(\tilde\Phi,\tilde\Psi)$ of nonhomogeneous wavelet systems is called a {\bf dual pair of nonhomogeneous wavelet frames} for $L^2(\R)$ if, for all $f\in L^2(\R)$, 
\be\label{dfl2}
f=\sum_{i=1}^r\sum_{k\in\Z}\<f,D^JT^k\tilde\phi^i\>_{L^2(\R)}\,D^JT^k\phi^i+\sum_{i=1}^s\sum_{j\geq J}\sum_{k\in\Z}\<f,D^jT^k\tilde\psi^i\>_{L^2(\R)}\,D^jT^k\psi^i
\,.
\ee
The convergence in (\ref{dfl2}), and everywhere in what follows, is in $L^2$-sense.

By way of illustration, we include here two (unpublished) results by the authors without proof. These results lead to a version of the OEP similar to the Han's version \cite[Theorem 9 and Corollary 10]{H10}. Apart from the fact that all the reasoning is done inside the Hilbert space $L^2(\R)$, our work does not go beyond the subtle Han's distributional approach.

\begin{prop}\label{ll45}
Consider the four finite sets of $L^2(\R)$ given in (\ref{fsphis}). Then the following assertions are equivalent:
\begin{enumerate}
\item
For some $J\in\Z$, $X_J(\Phi,\Psi)$ and $X_J(\tilde\Phi,\tilde\Psi)$ form a  dual pair of nonhomogeneous wavelet frames for $L^2(\R)$.
\item
For all $J\in\Z$, $X_J(\Phi,\Psi)$ and $X_J(\tilde\Phi,\tilde\Psi)$ form a  dual pair of nonhomogeneous wavelet frames for $L^2(\R)$.
\item
For some $J\in\Z$ and for all $f\in L^2(\R)$,
\be\label{dfl2r}
\begin{array}{r}
\ds\sum_{i=1}^r\sum_{k\in\Z}\<f,D^JT^k\tilde\phi^i\>_{L^2(\R)}\,D^JT^k\phi^i+\sum_{i=1}^s\sum_{k\in\Z}\<f,D^JT^k\tilde\psi^i\>_{L^2(\R)}\,D^JT^k\psi^i=
\\[2ex]
\ds=\sum_{i=1}^r\sum_{k\in\Z}\<f,D^{J+1}T^k\tilde\phi^i\>_{L^2(\R)}\,D^{J+1}T^k\phi^i
\end{array}
\ee
and the following ``completeness condition" is satisfied:
\be\label{dfl2r1}
\lim_{J\to\infty} \sum_{i=1}^r\sum_{k\in\Z}\<f,D^JT^k\tilde\phi^i\>_{L^2(\R)}\,D^JT^k\phi^i=f\,,\quad (f\in L^2(\R))\,.
\ee
\item
For all $J\in\Z$ and $f\in L^2(\R)$, (\ref{dfl2r}) is satisfied, and also the completeness condition (\ref{dfl2r1}) is verified.
\end{enumerate}
\end{prop}

\begin{theor}\label{ll45tx}
Let $\phi,\tilde\phi\in L^2(\R)$ be refinable functions. Let $V_0=\overline{\text{span}}\{T^k\phi:k\in\Z\}$, $\tilde V_0=\overline{\text{span}}\{T^k\tilde\phi:k\in\Z\}$ and $V_j:=D^jV_0$, $\tilde V_j:=D^j\tilde V_0$, $j\in\Z$.
Consider four finite sets of $L^2(\R)$ as given in (\ref{fsphis}) and such that 
$$
\Phi\subset V_0\,,\quad \Psi\subset V_1
\quad \text{and}\quad
\tilde\Phi\subset \tilde V_0\,,\quad \tilde\Psi\subset \tilde V_1\,,
$$
so that there exist masks $H_*^{D^{-1}\phi}$, $H_*^{\phi^i}$, $H_*^{D^{-1}\psi^l}$ relative to $V_0$ and masks $\tilde H_*^{D^{-1}\tilde\phi}$, $\tilde H_*^{\tilde\phi^i}$, $\tilde H_*^{D^{-1}\tilde\psi^l}$ relative to $\tilde V_0$.
Let us put
$$
\Theta(\om):=\sum_{i=1}^r H_*^{\phi^i}(\om)\,\overline{\tilde H_*^{\tilde\phi^i}(\om)}\,,\quad \text{for a.e. }\om\in\partial\D\,.
$$
Assume that 
\be\label{caSphi}
||\widehat{\phi^i_*}||_{l^2(\Z)}\cdot||\widehat{\tilde\phi^i_*}||_{l^2(\Z)}\in L^\infty(\partial\D)\,,\quad (i=1,\ldots,r)\,,
  \ee
\be\label{caSpsi}
||\widehat{\psi^l_*}||_{l^2(\Z)}\cdot||\widehat{\tilde\psi^l_*}||_{l^2(\Z)}\in L^\infty(\partial\D)\,,\quad (l=1,\ldots,s)\,.
\ee
Then, (\ref{dfl2r}) is satisfied for $J=0$ and all $f\in L^2(\R)$,
if and only if,
\be\label{oep1}
\begin{array}{r}
\ds\Theta(\om^2)\,H_*^{D^{-1}\phi}(\om)\,\overline{\tilde H_*^{D^{-1}\tilde\phi}(\om)}+\sum_{l=1}^s H_*^{D^{-1}\psi^l}(\om)\,\overline{\tilde H_*^{D^{-1}\tilde\psi^l}(\om)}=2\, \Theta(\om)\,,\\
\text{for a.e. }\om\in\sigma(\phi)\cap\sigma(\tilde\phi)\,,
\end{array}
\ee
and
\be\label{oep2}
\begin{array}{r}
\ds\Theta(\om^2)\,H_*^{D^{-1}\phi}(\om)\,\overline{\tilde H_*^{D^{-1}\tilde\phi}(-\om)}+\sum_{l=1}^s H_*^{D^{-1}\psi^l}(\om)\,\overline{\tilde H_*^{D^{-1}\tilde\psi^l}(-\om)}=0\,,\\
\text{for a.e. }\om\in\sigma(\phi)\text{ such that }-\om\in\sigma(\tilde\phi)\,.
\end{array}
\ee
\end{theor}

\begin{coro}[Oblique Extension Principle]\label{ll45cx}
Under the conditions of Theorem \ref{ll45tx} (so that (\ref{caSphi}) and (\ref{caSpsi}) are assumed), the following assertions are equivalent:
\begin{enumerate}
\item
For some $J\in\Z$, $X_J(\Phi,\Psi)$ and $X_J(\tilde\Phi,\tilde\Psi)$ form a  dual pair of nonhomogeneous wavelet frames for $L^2(\R)$.
\item
For all $J\in\Z$, $X_J(\Phi,\Psi)$ and $X_J(\tilde\Phi,\tilde\Psi)$ form a  dual pair of nonhomogeneous wavelet frames for $L^2(\R)$.
\item
(\ref{dfl2r1}), (\ref{oep1}) and (\ref{oep2}) are satisfied.
\end{enumerate}
\end{coro}

There are recent versions of the OEP for homogeneous wavelet systems $X(\Psi)$ of the form (\ref{wsxx}) too. For the sake of completeness, theorem \ref{oepams} below  reproduces an univariate version of the OEP due to Atreas, Melas and Stavropoulos \cite[Proposition 3.1]{AMS14}.

\begin{theor}[Oblique Extension Principle. Second version]\label{oepams}
Let $\phi\in L^2(\R)$ be a compactly supported refinable function satisfying:
\begin{enumerate}[(i)]
\item
$\hat\phi$ is continuous in a neighbourhood of the origin and verifies (\ref{ftpc0}).
\item
$||\hat\phi_*||_{l^2(\Z)}\in L^\infty(\partial\D)$.
\end{enumerate}
Let $V_0=\overline{\text{span}}\{T^k\phi:k\in\Z\}$ and $V_j:=D^jV_0$, $j\in\Z$.
Consider the finite set of $L^2(\R)$
$$
\Psi=\{\psi^1,\ldots,\psi^s\}\subset V_1\,.
$$
Then the following assertions are equivalent:
\begin{enumerate}[1.]
\item
$X(\Psi)$ is a tight wavelet frame with frame bound $1$ for $L^2(\R)$.
\item
There exists a measurable function $\Theta$ on $\partial\D$ such that
\begin{enumerate}[(a)]
\item
$\ds\lim_{j\to-\infty}\Theta(\om^{2^j})=1$ for a.e. $\om\in\partial\D$.
\item
One has 
\be\label{oep1x}
\begin{array}{r}
\ds\Theta(\om^2)\,H_*^{D^{-1}\phi}(\om)\,\overline{H_*^{D^{-1}\phi}(\om)}+\sum_{l=1}^s H_*^{D^{-1}\psi^l}(\om)\,\overline{H_*^{D^{-1}\psi^l}(\om)}=\Theta(\om)\,,\\
\text{for a.e. }\om\in\sigma(\phi)\,,
\end{array}
\ee
and
\be\label{oep2x}
\begin{array}{r}
\ds\Theta(\om^2)\,H_*^{D^{-1}\phi}(\om)\,\overline{H_*^{D^{-1}\phi}(-\om)}+\sum_{l=1}^s H_*^{D^{-1}\psi^l}(\om)\,\overline{H_*^{D^{-1}\psi^l}(-\om)}=0\,,\\
\text{for a.e. }\om\in\sigma(\phi)\text{ such that }-\om\in\sigma(\phi)\,.
\end{array}
\ee
\item
$\ds\int_{\partial\D}\Theta(\om)\,||\hat\phi_*(\om)||_{l^2(\Z)}^2\,\frac{d\omega}{2\pi}<\infty$.
\end{enumerate}
\end{enumerate}
\end{theor}
When $\Theta=1$, (\ref{oep1x}) and (\ref{oep2x}) together just coincide with the condition (\ref{ftpc0x}) in the UEP (how to adjust the support $\sigma(\phi)$ in (\ref{ftpc0x}) is now clear).

The equivalence between homogeneous and nonhomogeneous wavelet frames has been established in \cite{APS16}. 

Using (\ref{rmaskw}) and (\ref{rmaskwx}), the conditions (\ref{oep1x}) and (\ref{oep2x}) read
\be\label{oep1xa}
\begin{array}{r}
\ds\Theta(\om^2)\,\big|\<\widehat{[D^{-1}\phi]}_*(\om),\hat\phi_*(\om)\>_{l^2(\Z)}\big|^2+\sum_{l=1}^s \big|\<\widehat{[D^{-1}\psi^l]}_*(\om),\hat\phi_*(\om)\>_{l^2(\Z)}\big|^2=\Theta(\om)\,||\hat\phi_*(\om)||^2_{l^2(\Z)}\,,\\
\text{for a.e. }\om\in\partial\D\,,
\end{array}
\ee
and
\be\label{oep2xa}
\begin{array}{r}
\ds\Theta(\om^2)\,\<\widehat{[D^{-1}\phi]}_*(\om),\hat\phi_*(\om)\>_{l^2(\Z)}\,\<\hat\phi_*(-\om),\widehat{[D^{-1}\phi]}_*(-\om)\>_{l^2(\Z)}+\\
\ds+\sum_{l=1}^s \<\widehat{[D^{-1}\psi^l]}_*(\om),\hat\phi_*(\om)\>_{l^2(\Z)}\,\<\hat\phi_*(-\om),\widehat{[D^{-1}\psi^l]}_*(-\om)\>_{l^2(\Z)}=0\,,\\
\text{for a.e. }\om\in\partial\D\,.
\end{array}
\ee

It is clear that the role of the ONBs $\{L_{i}^{(0)}(x)\}_{i\in\I}$ and $\{K_{\pm,j}^{(0)}(x)\}_{j\in\J}$, the coefficients $\alpha_{i,j}^{s,l,k}$ depend on, in the spectral formulas (\ref{sesg}) of corollary \ref{coro10} is supplied  by the refinable function $\phi$ in the extension principles given in theorem \ref{tuep}, corollary \ref{ll45cx} and theorem \ref{oepams}, in particular, in formulas (\ref{oep1xa}) and (\ref{oep2xa}).


\section*{Acknowledgements}
This work was partially supported by research projects MTM2012-31439 and MTM2014-57129-C2-1-P (Secretar\'{\i}a General de Ciencia, Tecnolog\'{\i}a e Innovaci\'on, Ministerio de Econom\'{\i}a y Competitividad, Spain).


\end{document}